\documentclass[preprint,review,12pt]{elsarticle}

\usepackage{amssymb,amscd,amsfonts,amsmath,mathrsfs,graphicx,amsthm, listings, xcolor}
\usepackage[all]{xy}

\newcommand{\Hom}{{\textnormal{Hom}}}
\newcommand{\Ext}{{\textnormal{Ext}}}
\newcommand{\im}{{\textnormal{Im }}}

\newcommand{\HH}{{\textnormal{HH}}}
\newcommand{\rad}{{\textnormal{rad}}}
\newcommand{\C}{{\mathcal C}}
\newcommand{\B}{{\mathcal B}}

\newtheorem{teo}{Theorem}
\newtheorem{prop}[teo]{Proposition}
\newtheorem{lema}[teo]{Lemma}

\newproof{pf}{Proof}
\newdefinition{deft}{Definition}
\newtheorem{obs}[deft]{Remark}
\newtheorem{ex}[deft]{Example}

\lstdefinelanguage{Sage}{
  keywords={if,else,for,while,return,break,continue,
            def,class,import,from,as},
  morekeywords={ZZ,QQ,RR,CC,GF,PolynomialRing},
  sensitive=true,
  morecomment=[l]\#,
  morestring=[b]",
}

\journal{Linear Algebra and its Applications}

\begin{document}

\begin{frontmatter}



\title{Projective resolutions of simple modules and Hochschild cohomology for incidence algebras} 

 \author{V. Bekkert}

\author{J. W. MacQuarrie}
 

\author{J. Marques} 

\begin{abstract}
We give a practical, algorithmic method to calculate minimal projective resolutions of simple modules for a finite dimensional incidence $k$-algebra $\Lambda$, where $k$ is a field.  We apply the method to the calculation of Ext groups between simple $\Lambda$-modules, Hochschild cohomology groups $\HH^i(\Lambda, \Lambda)$, and singular cohomology groups of finite $T_0$ topological spaces with coefficients in $k$.
\end{abstract}

%


\end{frontmatter}



\label{sec1}
\section{Introduction}
When $\Lambda$ is a finite dimensional algebra, a projective resolution of a finitely generated $\Lambda$-module is \emph{minimal} if each map of the resolution, with codomain restricted to the kernel of the previous one, is a projective cover \cite[Definition I.5.7(b)]{ASS06}.  We present here a practical algorithm yielding minimal projective resolutions 
of the simple modules for a finite dimensional incidence algebra $\Lambda$ (see the first paragraph of Section \ref{section construction of resolution} for the formal definition)   -- Theorem \ref{projectiveresolution}.  The computations require only very basic and computationally efficient operations in linear algebra: calculations of kernels of linear maps, intersections and sums of subspaces, and the calculation of bases for complements of subspaces.  An algorithm due to Bongartz and Butler (\cite{B83}, cf.\ \cite[Proposition 3.5]{C89}) calculates projective resolutions in the same circumstances, but they are not necessarily minimal and the computation is more cumbersome, working with ideals of $\Lambda$ rather than just linear algebra.  

Our algorithm can be used to calculate the groups $\Ext^i_{\Lambda}(S,T)$, where $S,T$ are simple $\Lambda$-modules -- Proposition \ref{extgroups}.  The main application of this fact comes via a result of Cibils \cite[Proposition 2.1]{C89} that gives the Hochschild cohomology groups $\HH^i(\Lambda, \Lambda)$ of $\Lambda$ in terms of Ext groups between two simple modules for a related incidence algebra $\Lambda^*$, and hence (via Proposition \ref{extgroups}) the algorithm gives a practical way to calculate the Hochschild cohomology groups of an incidence algebra -- Theorem \ref{hochschildcohomology}.  The computation of Hochschild cohomology groups of incidence algebras has been considered in the literature \cite{GS83, C89, H89, IZ90, GR01}.  For instance, \cite{GS83} explains how these groups can be calculated as the cohomology groups of simplicial complexes, while in \cite{GR01}, explicit formulas for the Hochschild cohomology groups of important particular classes of incidence algebras are given.  However, there is to our knowledge no simple, practical algorithm available to calculate the Hochschild cohomology groups of an arbitrary finite dimensional incidence algebra, as given by Theorem \ref{hochschildcohomology}.

The algorithm has an application to topology -- Remark \ref{obs top spaces}.  Namely, finite $T_0$ topological spaces are in natural bijection with finite posets, and via results due to McCord \cite{M66} and Gerstenhaber and Schack \cite{GS83}, it follows that Theorem \ref{hochschildcohomology} also yields a practical computation of the singular cohomology groups of a finite $T_0$ topological space with coefficients in $k$. 

The results presented here have theoretical applications: in future work, we will show how they provide new ways to reduce the poset without altering the Hochschild cohomology groups.

\section{Construction of the resolution}\label{section construction of resolution}

Let $k$ be a field and $X$ be a poset, always finite.  From $X$, we construct an acyclic quiver $Q$ -- the Hasse diagram of $X$ -- as follows: the vertices of $Q$ are the elements of $X$, and there is an arrow from the vertex $a$ to the vertex $b$ if, and only if, $a<b$ and there is no $c\in X$ such that $a<c<b$.  The incidence algebra of $X$ is $\Lambda = kQ/I$, where $I$ is the two-sided ideal of $kQ$ generated by all differences of parallel paths (i.e., paths with the same source and target).  Thus whenever $a\leqslant b$ in $X$, any two paths from $a$ to $b$ in $Q$ correspond to the same element of $\Lambda$, which we denote by $c_{a,b}$.  For each $x\in X$, we denote by $S_x$ and $P_x$ the simple and indecomposable right projective $\Lambda$-modules associated to $x$. We denote the sets of immediate successors and predecessors of $x$ in $X$  by $x^+$ and $x^-$, respectively.

\medskip 

Given a poset $X$ and $x\in X$, for each $i\geqslant 0$, we will define a set of what we call ``$i$-cycles'', which will index the indecomposable summands of the projective module $P_i$ appearing in a minimal projective resolution
$$\xymatrix@R=0.8em{ \cdots\ar[r]& \ar[r] P_1 \ar[r] & P_0 \ar[r] & S_x \ar[r] & 0}$$
of the simple $\Lambda$-module $S_x$.

\subsection{Definition of $i$-cycles}\label{subsec icycle def}
Given a finite set $L$, let $kL$ be the $k$-vector space with basis $L$, and given $w = \sum_{l\in L}\alpha_ll\in kL$ ($\alpha_i\in k$), let $supp(w) = \{l\in L\,|\,\alpha_l\neq 0\}$.  

We define the set $\C^i$ of $i$-cycles for each $i \geqslant 0$ and a fixed element $x \in X$. Let 
	$$\C^{0}=\{x\}\hbox{ and }\C^{1} = \{(x,y)\,|\,y\in x^+\}.$$  
Define the linear map $\partial_{1} : k\C^{1}\to k\C^{0}$ on the basis $\C^{1}$ by $(x,y)\mapsto x$. For convenience we also define $k\C^{-1} = 0$ and $\partial_{0} : k\C^{0}\to k\C^{-1}$ to be the zero map. We define the $i$-cycles $\C^i$ for $i\geqslant 2$ recursively: For fixed $i\geqslant 1$, suppose that $\C^i\subseteq \ker\partial_{i-1}\times X$ and the map $\partial_i : k\C^i\to k\C^{i-1}$ are already defined. For each $z\in X$, consider the $k$-subspace $(\ker\partial_{i})_z$	of $\ker\partial_{i}$ given by 
	$$\{w=\sum_{t=1}^n \lambda_t(w_t,z_t)\in \ker\partial_{i}\ |\ z_t < z \textnormal{ for every } (w_t,z_t)\in supp(w) \}.$$
Define also $(\ker\partial_{i})_{z^-}:=\sum_{z'\in z^-}(\ker\partial_{i})_{z'}$. Let ${\mathcal{D}^{i+1}_z}$ be a basis of a complement of $(\ker\partial_{i})_{z^-}$ in $(\ker\partial_{i})_{z}$ and define 
	$$\B^{i+1}_z = \{(w,z)\,|\,w\in {\mathcal{D}^{i+1}_z}\}.$$ 
The set of $(i+1)$-cycles is defined to be
	$$\C^{i+1}:=\bigcup_{z\in X}\B^{i+1}_z.$$
Define further the linear map $\partial_{i+1}:k\C^{i+1}\rightarrow k\C^i$ on the given basis by 
	$$\partial_{i+1}((w,z))=w$$

\begin{obs}\label{obs z' are less than z}
We make explicit a basic but fundamental observation about the $i$-cycles, because it will be used a great deal in what follows.  Namely, if ever $(w,z)\in k\mathcal{C}^{i+1}$ for some $i\geqslant 1$  and $(w',z')\in k\mathcal{C}^i$ is in $supp(w)$ then, $w$ being an element of $(\ker\partial_{i})_z$, necessarily $z'<z$.
\end{obs}

\begin{obs}\label{obs useful basis of ker partial}
Indexing the elements of the poset $X$ by $1, \hdots, |X|$ in such a way that the index of any successor of any $z$ in $X$ is greater than that of $z$, an easy induction on $z$ shows that for each $i\geqslant 1$ and each $z\in X$, we can choose a basis $B$ of $(\ker\partial_{i})_{z}$ such that for every $b\in B$, $(b,z_j)\in \C^{i+1}$ for some $z_j\leqslant z$.
\end{obs}

\begin{ex}
Let $X$ be the poset with the following Hasse diagram
	$$
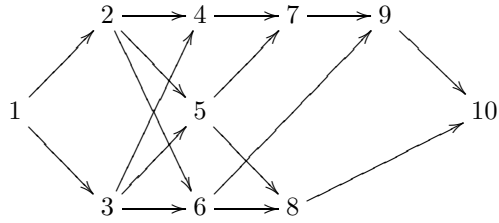
\xymatrix{ & 2\ar[r]\ar[rd]\ar[rdd] & 4\ar[r] & 7\ar[r] & 9\ar[rd] & \\ 1\ar[rd]\ar[ru]& & 5\ar[rd]\ar[ru] & & & 10 \\
	& 3\ar[ruu]\ar[ru]\ar[r] & 6\ar[r]\ar[rruu] & 8\ar[rru] & }$$
Fixing $x=1$, we have
	$$\C^{0} = \{1\}\hbox{ and } \C^1 =  \{(1,2),(1,3)\}.$$
We have $(\ker \partial_1)_z = 0$ for $z<4$ and $(\ker \partial_1)_z = k\{ r = (1,2)-(1,3)\}$ for $z\geqslant 4$, so that we may take
	$$\C^2 = \{(r,4), (r,5), (r,6)\}$$  
and 
\begin{align*}
\partial_2 : k\C^2 & \to k\C^{1} \\
(r,j) &\mapsto r\,,\quad j\in \{4,5,6\}.
\end{align*}
The kernel of $\partial_2$ has dimension two and basis 
	$$\{u = (r,4)-(r,5), v = (r,5)-(r,6)\}.$$
We have 
	$$(\ker \partial_2)_{7}=k\{ u\}\,,\, (\ker \partial_2)_{7^-} = 0\hbox{ and hence }\B_7^3 = \{(u,7)\}.$$
Similarly $\B_8^3 = \{(v,8)\}$. 
We have $(\ker \partial_2)_9=k\{u,v\} $ and $(\ker \partial_2)_{9^-}=k\{ u\}$. A possible complement of $(\ker \partial_2)_{9^-}$ in $(\ker \partial_2)_{9}$ is $k\{ u+v\}$, so we may take $\B^3_9=\{(u+v,9)\}$. Because $(\ker \partial_2)_{10}=k\{u,v\}=(\ker \partial_2)_{10^-}$, we have $\B^3_{10}=\varnothing$. Now 
	$$\C^3= \B^3_7 \cup \B^3_8\cup \B^3_9 = \{(u,7),(v,8), (u+v,9)\}.$$     
The map $\partial_3 : k\C^3  \to k\C^{2}$ sends $(u,7)$ to $u$, $(v,8)$ to $v$, $(u+v,9)$ to $u+v$.  The kernel of $\partial_3$ has dimension one and basis $\{q = (u,7)+(v,8)- (u+v,9)\}$.  Now $k\C^4_{10}=k\{ q\}$ and $k\C^4_{10^{-}} = 0$, so $\B^4_{10}=\{(q,10)\}$ and hence $\C^4=\{(q,10)\}$.  Finally, the map $\partial_4 : k\{ (q,10)\} \to k\{ (u,7), (v,8), (u+v, 9) \}$ sends $(q,10)$ to $q = (u,7)+(v,8)- (u+v,9)$ and hence is injective, so that $\C^5 = \varnothing$.  
\end{ex}

\begin{lema}\label{lema1}
Let $X$ be a poset and $x\in X$.  With the definitions as above, the sequence 
	$$C_{\bullet}: \xymatrix@R=0.8em{ \cdots \ar[r]^-{\partial_3} & k\C^2 \ar[r]^-{\partial_2} & k\C^{1} \ar[r]^-{\partial_{1}} & k\C^{0}}$$
is a chain complex.
\end{lema}
\begin{pf}
The maps $\partial_i$ are defined to be linear, so it is enough to check that $\partial_i(\partial_{i+1}(w,z)) = 0$ for $(w,z)\in \mathcal{C}^{i+1}$.  But $w\in \ker(\partial_i)$ by the definition of $\mathcal{C}^{i+1}$, and so
$$\partial_i(\partial_{i+1}(w,z)) = \partial_i(w) = 0,$$
as required.
\end{pf}
The above complex need not be exact.  For example, with the poset
	$$
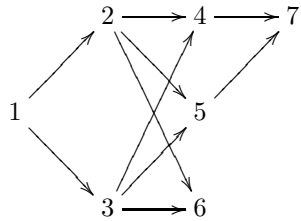
\xymatrix{ & 2\ar[r]\ar[rd]\ar[rdd] & 4\ar[r] & 7  \\ 1\ar[rd]\ar[ru]& & 5\ar[ru] &   \\
	& 3\ar[ruu]\ar[ru]\ar[r] & 6 &  }$$
and $x=1$, the image of $\partial_3$ has codimension $1$ in the kernel of $\partial_2$.

\subsection{Definition of the resolution}\label{subsec resolution def}
Denote by $\Lambda_0$ the semisimple subalgebra of $\Lambda$ generated by the stationary paths $c_a = c_{a,a}$ ($a\in X$).  Denote by $\delta_{z,a}$ the usual delta function, which is $1$ when $z=a$ and $0$ otherwise.  For each $i\geqslant 0$, we give $k\C^i$ the structure of a right $\Lambda_0$-module, whose action on the basis $\C^i$ is as follows: 
\begin{align*}
x\cdot c_a = \delta_{x,a}x, & \quad i=0, \\
(w,z)\cdot c_a = \delta_{z,a}(w,z),& \quad i> 0.
\end{align*}
We may thus define the vector space $k\C^i\otimes_{\Lambda_0}\Lambda$, which is a right $\Lambda$-module, treating $\Lambda$ as a $\Lambda_0\hbox{-}\Lambda$-bimodule in the obvious way.

Given $z\in X$, denote by $q_z = \sum_{u\leqslant z}c_{u,z}\in \Lambda$ the sum of the paths ending at $z$.  For each $i\geqslant 1$, we define the map
$d_i: k\C^i\otimes_{\Lambda_0}\Lambda\longrightarrow k\C^{i-1}\otimes_{\Lambda_0}\Lambda$ on the generators $\C^i\otimes 1$ of $k\C^i\otimes_{\Lambda_0}\Lambda$ as follows:
	$$\begin{array}{c c l}
	d_i((w,z)\otimes 1)&=&w\otimes q_z.
	\end{array}$$
Note that when tensoring elements, we use the symbol $\otimes$ rather than $\otimes_{\Lambda_0}$ (writing for instance $w\otimes q_z$ rather than $w\otimes_{\Lambda_0} q_z$).  The $d_i$ are easily checked to be well-defined homomorphisms of right $\Lambda$-modules.  

By Remark \ref{obs z' are less than z}, for any $(w,z)\in \C^n$ (some $n$) we have that $z'<z$ for every $(w',z') \in supp(w)$. Thus, $w\cdot E = w$ for $E=\sum_{e\leqslant z} c_e$ in $\Lambda_0$. Note that if ever $e,z,z'\in X$ with $e\leqslant z'\leqslant z$, then in $\Lambda$ we have $c_eq_{z'}q_z = c_{e,z} = c_eq_z$. Let $(w',z')$ in $supp(w)$ and $E'=\sum_{e\leqslant z'} c_e$ in $\Lambda_0$. Then in $k\mathcal{C}^n\otimes_{\Lambda_0}\Lambda$ we have
    $$w'\otimes q_{z'}q_z = w'E'\otimes q_{z'}q_z = w'\otimes E'q_{z'}q_z = w'\otimes E'q_z = w'E'\otimes q_z = w'\otimes q_z.$$
We use this observation in the proofs of Lemma \ref{lemma ds give complex} and Theorem \ref{projectiveresolution} below.
\begin{lema}\label{lemma ds give complex}
The sequence
	$$\xymatrix@R=0.8em{ \cdots \ar[r] & k\C^2\otimes_{\Lambda_0}\Lambda\ar[r]^-{d_2} & k\C^{1}\otimes_{\Lambda_0}\Lambda\ar[r]^-{d_{1}} & k\C^{0}\otimes_{\Lambda_0}\Lambda\ar[r] & 0}$$
is a chain complex.
\end{lema}

\begin{pf}
We check that $d_{i-1}d_i = 0$ for $i\geqslant 2$ fixed.  Consider a generator $(w,z)\otimes 1$ of $k\C^i\otimes_{\Lambda_0} \Lambda$, write $w = \sum_{(w',z')\in \C^{i-1}} \alpha_{(w',z')}(w',z')$ recalling (cf.\ Remark \ref{obs z' are less than z}) that if $\alpha_{(w',z')}\neq 0$, then $z'<z$.  We have
\begin{align*}
d_{i-1}d_i((w,z)\otimes 1) & = d_{i-1}(w\otimes q_z) \\
					       & = d_{i-1}(\sum\alpha_{(w',z')}(w',z')\otimes 1)q_z \\
            			   & = \sum\alpha_{(w',z')}d_{i-1}((w',z')\otimes 1)q_z \\
            			   & = \sum\alpha_{(w',z')}w'\otimes q_{z'}q_z \\
            			   & = \sum\alpha_{(w',z')}w'\otimes q_z \\
            			   & = (\sum\alpha_{(w',z')}w')\otimes q_z \\
            			   & = \partial_{i-1}\partial_i((w,z))\otimes q_z \\
            			   & = 0
\end{align*}
by Lemma \ref{lema1}.
\end{pf}

\begin{teo}\label{projectiveresolution}
Let $X$ be a poset and $x\in X$. With definitions as above, the complex
	$$\xymatrix@R=0.8em{ \cdots \ar[r] & k\C^2\otimes_{\Lambda_0}\Lambda\ar[r]^-{d_2} & k\C^{1}\otimes_{\Lambda_0}\Lambda\ar[r]^-{d_{1}} & k\C^{0}\otimes_{\Lambda_0}\Lambda\ar[r] & S_x\ar[r] & 0}$$
where, writing $S_x = \langle s\rangle$, the final non-zero map sends $x\otimes c_{x,a}$ to $\delta_{x,a}s$, is a minimal projective resolution of the simple $\Lambda$-module $S_x$.
\end{teo}
	
\begin{pf}
It is a standard fact (and easily checked) that if $S_z$ is the simple right $\Lambda$-module at the vertex $z$ of $X$, the $\Lambda$-module $S_z\otimes_{\Lambda_0}\Lambda$ is its projective cover.  From this, it follows that the modules $k\C^{i}\otimes_{\Lambda_0}\Lambda$ are projective for every $i$.  That the sequence is exact at $S_x$ and $k\C^{0}\otimes_{\Lambda_0}\Lambda$ is clear.
 
Fix $i\geqslant 1$.  We check that $\ker d_i\subseteq \im d_{i+1}$.  The kernel of $d_{i}$, as with any right $\Lambda$-module, has a basis of elements $g$ such that $g = g\cdot c_z$ for some $z\in X$.  Fix some such $g$. Hence $g = \sum_{j=1}^n\alpha_j(w_j,z_j)\otimes c_{z_j,z}$ for some $(w_j,z_j)\in \C^i$ with $z_j \leqslant z$ and $\alpha_j\in k$.  Denoting by $v$ the element $\sum_{j=1}^n \alpha_j(w_j,z_j) \in k\C^{i}$ we have $g = v\otimes q_z$, and 
	$$0 = d_i(g) = \sum_{j=1}^n\alpha_jw_j\otimes q_{z_j}c_{z_j,z} = (\sum_{j=1}^n\alpha_jw_j)\otimes q_z = \partial_i(v)\otimes q_z.$$
There is a well-defined linear map $k\C^{i-1}\otimes_{\Lambda_0} \Lambda \to k\C^{i-1}$ defined on basis elements by $b\otimes c_{u,y}\mapsto b\cdot c_u$, which sends $0 = (\sum_{j=1}^n\alpha_jw_j)\otimes q_z$ to $\partial_i(v)$, and hence $v\in \ker \partial_i$.

With the same $g$ as above, we claim that if $(w_j,z_j)$ appears with non-zero coefficient $\alpha_j$, then in fact $z_j < z$.  Reindexing if necessary, we may suppose that $z_1, \hdots, z_a$ are equal to $z$ (some $a\in \{0, \hdots, n\}$), while $z_{a+1}, \hdots z_n$ are strictly less than $z$.  The elements $w_1, \hdots, w_a$ must therefore be distinct (hence linearly independent) elements of ${\mathcal{D}^{i}_z}$, while $w_{a+1}, \hdots, w_n\in (\ker \partial_{i-1})_{z^-}$.  Since $\partial_i(v) = 0$,
	$$\sum_{j=1}^a\alpha_j w_j = -\sum_{j=a+1}^n\alpha_j w_j \in k{\mathcal{D}^{i}_z} \cap (\ker \partial_{i-1})_{z^-} = 0,$$
and hence $\alpha_j = 0$ for $j\in \{1, \hdots, a\}$, as claimed.

It follows from the above claim that $v\in (\ker \partial_i)_z$.  By Remark \ref{obs useful basis of ker partial}, there are elements $(u_1,z_1'), \hdots, (u_m,z_m')$ of $k\C^{i+1}$ with each $z_l'\leqslant z$ and such that $\{u_1, \hdots, u_m\}$ is a basis of $(\ker \partial_i)_z$.  Write
	$$v = \sum_{l=1}^m\beta_lu_l$$
for some $\beta_l\in k$ and consider $h = \sum_l \beta_l(u_l,z_l')\otimes c_{z_l',z}\in k\C^{i+1}\otimes_{\Lambda_0} \Lambda$.  Then 
	$$d_{i+1}(h) = \sum\beta_lu_l\otimes q_{z_l'}c_{z_l',z} =  (\sum \beta_lu_l)\otimes q_z = v\otimes q_z = g,$$
completing the proof that $\ker d_i\subseteq \im d_{i+1}$ and hence that the sequence is exact.

To show minimality, we will check that $\ker d_i\subseteq \rad (k\C^i\otimes_{\Lambda_0} \Lambda)$ \cite[Lemma VIII.2.1]{A97}.  We use again that if $g = \sum_{j=1}^n\alpha_j(w_j,z_j)\otimes c_{z_j,z}$ is in $\ker d_i$, then $z_j<z$.   Hence 
	$$\alpha_j(w_j,z_j)\otimes c_{z_j,z} = (\alpha_j(w_j,z_j)\otimes c_{z_j})\cdot c_{z_j,z}\in (k\C^i\otimes_{\Lambda_0} \Lambda)\cdot J(\Lambda) = \rad (k\C^i\otimes_{\Lambda_0} \Lambda)$$
for each $j$, so that $g\in \rad (k\C^i\otimes_{\Lambda_0} \Lambda)$ as required.
\end{pf}

\subsection{Application to Ext groups}\label{subsec Ext groups}
Our next result is a computationally easy description of the $\Ext$-groups for simple modules in an incidence algebra.  See \cite[Proposition 3.6]{C89} for a different description, which uses the Bongartz-Butler projective resolution \cite[Section 1.1]{B83}. Let $X$ be a partially ordered set with incidence algebra $\Lambda$, and $x,b\in X$.  We give a formula for $\Ext_{\Lambda}^n(S_x,S_b)$.

Consider the $i$-cycles constructed in Section \ref{subsec icycle def} for $x\in X$.  We use the size $|\B^i_b|$ of the sets $\B^i_b$ ($i\geqslant 2$), augmenting the definition to include $i=0,1$ as follows:
   	$$\begin{array}{r c l}
    |\B^{0}_b|=\begin{cases}
		1,\textnormal{ if } b=x \\
		0,\textnormal{ otherwise.}
	\end{cases} &\textnormal{ and } &
	|\B^{1}_b|=\begin{cases}
	1,\textnormal{ if } b\in x^+\\
	0,\textnormal{ otherwise.}
	\end{cases}
	\end{array}$$

\begin{prop}\label{extgroups}
Let $\Lambda=kX/I$ be an incidence algebra. Let $x,b$ be elements of $X$ and $S_x,S_b$ the corresponding simple $\Lambda$-modules.  For $i\geqslant 0$, we have
	$$\dim_k\Ext_{\Lambda}^i(S_x,S_b) = |\B^{i}_b|.$$
\end{prop}
\begin{pf}
By Theorem \ref{projectiveresolution}, the complex below is a deleted projective resolution of $S_x$.
	$$\xymatrix@R=0.8em{ \cdots \ar[r] & k\C^2\otimes_{\Lambda_0}\Lambda\ar[r]^-{d_2} & k\C^{1}\otimes_{\Lambda_0}\Lambda\ar[r]^-{d_{1}} & k\C^{0}\otimes_{\Lambda_0}\Lambda\ar[r] & 0}$$ 	
Applying the contravariant functor $\Hom_{\Lambda}(-,S_b)$ and denoting by $d^*$ the map $(-\circ d)$ as standard, we obtain the complex
	$$\xymatrix@R=0.8em{ 0 \ar[r]& \Hom_{\Lambda}(k\C^{0}\otimes_{\Lambda_0}\Lambda,S_b) \ar[r]^-{d_{1}^*} & \Hom_{\Lambda}(k\C^{1}\otimes_{\Lambda_0}\Lambda,S_b) \ar[r]^-{d_{2}^*} & \cdots } $$
whose homology groups are $\ker d_{i+1}^*/\im d_{i}^* = \Ext_{\Lambda}^i(S_x,S_b)$, $i\geqslant 0$.  If $x\not\leqslant b$ then, for each $i\geqslant 0$, $\Ext_{\Lambda}^i(S_x,S_b) = 0$ and $|\B^{i}_b|=0$. Suppose now that $x\leqslant b$. We have that $k\C^i\otimes_{\Lambda_0}\Lambda\cong \bigoplus_{(w,z)\in \C^i} P_z$. Therefore, 
	$$\dim\Hom_{\Lambda}(k\C^{i}\otimes_{\Lambda_0}\Lambda,S_b) = |\{(w,z)\in \C^i\,:\,z=b\}| = |\B^i_b|.$$ 
The proposition will thus follow from the observation that the maps $d_i^*$ are zero.  To see this, we first claim that for any $l>0$, an element $\rho$ of $\Hom_{\Lambda}(k\C^{l}\otimes_{\Lambda_0}\Lambda,S_b)$ must send elements of the form $y\otimes \lambda$ with $y\in \C^l$ and $yc_b=0$, to zero: if $y = yc_a\in \mathcal{C}^l$ (some $a$) and $\rho(y\otimes \lambda)\neq 0$, then
$$0 \neq \rho (y\otimes \lambda) = \rho (y\otimes c_a\lambda)c_b = \rho(y\otimes 1)c_a\lambda c_b,$$
implying that $a = b$ as claimed, because non-trivial paths act as $0$ on $S_b$.  Fix $\rho\in \Hom_{\Lambda}(k\C^{i-1}\otimes_{\Lambda_0}\Lambda,S_b)$.  Then 
	$$d_i^*(\rho)((w,b)\otimes 1)=\rho d_i((w,b)\otimes 1) = \rho(w\otimes q_b) = 0$$
because $wc_b=0$ (Remark \ref{obs z' are less than z}).
\end{pf}
	
\section{Hochschild cohomology}
Recall that $\Lambda$-bimodules are in natural correspondence with right $\Lambda\otimes_k \Lambda^{op}$-modules ($\Lambda^{op}$ the opposite algebra of $\Lambda$), the right action on the bimodule $M$ being given by $m(\lambda\otimes \lambda') = \lambda'm\lambda$. The Hochschild cohomology groups $\HH^i(\Lambda, M)$ with coefficients in the $\Lambda$-bimodule $M$ are, by definition, the groups $\Ext_{\Lambda\otimes_k \Lambda^{op}}^i(\Lambda,M)$. There are several projective resolutions that allow the computation of Hochschild cohomology groups \cite[Section 1.5]{H89}, \cite[Section 2]{Hoc45}. For our purposes, the resolution \cite[Lemma 1.1]{C89} is of particular importance, thanks to its connection with $\Ext$-groups.
	
As in \cite{C89}, we extend the poset $X$ as follows: let $x^*$ and $y^*$ be two points not in $X$. Define $X^*=X\cup\{x^*,y^*\}$ with the same ordering on the elements of $X$ and with $x^*<x<y^*$  for every $x\in X$. The next result, due to Cibils, allows us to calculate the Hochschild cohomology groups using the tools developed in the previous sections.
	
\begin{teo}[{\cite[p.225]{C89}}]\label{exthochschild}
Let $X$ be a poset and $X^*$ its extension as above, with $\Lambda$ and $\Lambda^{*}$ the respective incidence algebras. Then
	$$\HH^i(\Lambda, \Lambda)\cong \Ext_{\Lambda^*}^{i+2}(S_{x^*},S_{y^*})$$
for all $i\geqslant 1$, where $S_{x^*}$ and $S_{y^*}$ are the respective simple modules in $\Lambda^*$-mod.
\end{teo} 	

\begin{teo}\label{hochschildcohomology}
Let $\Lambda=kX/I$ be an incidence algebra. Then 
	$$\dim\HH^i(\Lambda,\Lambda) = 	
    \begin{cases}
        |\B^{2}_{y^*}|+1, \quad i=0, \\
		|\B^{i+2}_{y^*}|\;\quad,\quad i>0.  
	\end{cases}$$
\end{teo}	

\begin{pf}
The claim for $i>0$ follows from Proposition \ref{extgroups} and Theorem \ref{exthochschild}, so it remains to check for $i=0$. For any finite dimensional algebra $\Lambda$, $\HH^0(\Lambda,\Lambda)$ is isomorphic to the center $Z(\Lambda)$ \cite[p.9]{W19}.  But when $\Lambda = kQ/I$ with $Q$ acyclic, then $Z(\Lambda)\cong k^s$ where $s$ is the number of connected components of $Q$, so in our case, $\dim\HH^0(\Lambda, \Lambda)$ is simply the number of connected components $s$ of $X$.  Our task is thus to check that $|\B^{2}_{y^*}| = s-1$.  Recall that
	$$\partial_1 : k\{ (x^*,x_i)\,|\,x_i\in (x^*)^+\} \to k\{ x^*\}$$
is given on the basis as projection on the first coordinate, and that $|\B^{2}_{y^*}|$ is the codimension of $\sum_{y\in (y^*)^-}\ker (\partial_1)_y$ in $\ker (\partial_1)_{y^*} = \ker \partial_1$.  

Denote the connected components of $X$ as $X_1, \hdots, X_s$.  For $j\in \{1, \hdots, s\}$ and any proper subset $T$ of $X_j\cap (x^*)^+$, there must be some $y\in X_j\cap (y^*)^-$ having both an element $z$ of $T$ and an element $z'$ of $T' = (X_j\cap (x^*)^+)\setminus T$ as predecessor, because otherwise $X_j$ is not connected.  Beginning with $T = \{z\}$ with $z$ an arbitrary element of $X_j\cap (x^*)^+$ and using the above observation repeatedly, in the obvious way, until we have every element of $(X_j\cap (x^*)^+)$, we obtain $|(X_j\cap (x^*)^+)|-1$ linearly independent elements of $\sum_{y\in X_j\cap (y^*)^-}\ker(\partial_1)_y$, which must therefore be a basis of this space.  Noting further that
	$$(\ker \partial_1)_{(y^*)^-} = \sum_{j=1}^s
    \sum_{y\in X_j\cap (y^*)^-}\ker(\partial_1)_y =
    \bigoplus_{j=1}^s  
    \sum_{y\in X_j\cap (y^*)^-}\ker(\partial_1)_y
    $$
has dimension $|(x^*)^+|-s$ by the above calculations, we obtain
	$$|\B^{2}_{y^*}| = \dim \ker(\partial_1)_{y^*} - \dim \ker(\partial_1)_{(y^*)^-} = (|(x^*)^+| - 1) - (|(x^*)^+|-s) = s-1,$$
as required.
\end{pf}
\begin{obs}\label{obs top spaces}
There is a natural correspondence between finite posets and finite $T_0$ topological spaces (that is, finite topological spaces in which given two distinct points, there is an open set containing one of them but not the other).  Gerstenhaber and Schack \cite[p.148]{GS83} show that the Hochschild cohomology groups of the incidence algebra $\Lambda = kX$ of the poset $X$ are isomorphic to the simplicial cohomology groups with coefficients in $k$ of the simplicial complex obtained from $X$ by taking the $n$-simplices to be chains in $X$ of length $n$.  But McCord \cite[Theorem 1]{M66} has shown that these groups are isomorphic to the singular cohomology groups of the topological space corresponding to $X$. Thus Theorem \ref{hochschildcohomology} can be used to efficiently calculate the singular cohomology groups of finite $T_0$ topological spaces with coefficients in $k$.
\end{obs}

\begin{obs}\label{obs eficience}
    A basic implementation of an algorithm calculating the Hochschild cohomology groups of an incidence algebra using the $i$-cycles algorithm can be found in \cite{BMMGithubRepo}, where both pseudocode and a functional algorithm in SAGE are available.  The $i$-cycles algorithm is considerably faster than the very few algorithms available in computational algebra systems.  One example is the function ``\texttt{CompactProjectiveResolution}'' in Magma \cite{magma}, which calculates a minimal projective resolution of a module for a given algebra.  The code available in \cite{BMMGithubRepo} can be used to compare the run-times of our rudimentary implementation of the $i$-cycles algorithm in SAGE with that of CompactProjectiveResolution. Having copy and pasted the functions from \cite[compare\_algorithms.sage]{BMMGithubRepo} into a SAGE command line, the following code takes 10 random posets with 30 vertices, and gives the mean times in seconds for the two algorithms:
\end{obs}
\begin{lstlisting}[language=Python]
posets_list = [posets.RandomPoset(30, uniform(0, 1)) for _ in range(10)]
iCyclesTimes = []; CPRTimes = []
for X in posets_list:
    iCyclesTimes.append(compute_time(iCycles, X, 0))
    Send_poset_to_Magma(X)
    CPRTimes.append(compute_time(magma.eval,"PR:=CompactProjectiveResolution(S_X,m_X);"))

print("Mean iCycles time:", mean(iCyclesTimes), "seconds")
print("Mean CompactProjectiveResolution time:", mean(CPRTimes), "seconds")
\end{lstlisting}

\noindent Output:
\begin{lstlisting}
Mean iCycles time: 0.03402 seconds
Mean CompactProjectiveResolution time: 420.43140 seconds
\end{lstlisting}

\medskip 

\section{Acknowledgments}
Funding: The second author was partially supported by CNPq Universal Grant 402934/2021-0, CNPq Produtividade 1D Grant 303667/2022-2, and FAPEMIG Universal Grants APQ-00971-22 and APQ-03491-25. The third author was supported by CAPES Doctoral Grant 88887.465442/2019-00.


\begin{thebibliography}{00}

\bibitem{A97}
I.~Assem.
\newblock {\em Algebre et modules: Cours et exercices}.
\newblock Les Presses de l'Université d Ottawa, Ontario, Canada, 1997.

\bibitem{ASS06}
I.~Assem, D.~Simson, and A.~Skowro\'{n}ski.
\newblock {\em Elements of the representation theory of associative algebras.
  {V}ol. 1}, volume~65 of {\em London Mathematical Society Student Texts}.
\newblock Cambridge University Press, Cambridge, 2006.
\newblock Techniques of representation theory.


\bibitem{BMMGithubRepo}
V.~Bekkert, J.W.~MacQuarrie, and J.~Marques. {\em IncidenceAlgebraHH}.
\newblock Github, 2026.
\newblock https://github.com/JohnMacQuarrie/IncidenceAlgebraHH


\bibitem{B83}
K.~Bongartz.
\newblock Algebras and quadratic forms.
\newblock {\em J. London Math. Soc.}, 28:461--469, 1983.

\bibitem{magma}
Wieb Bosma, John Cannon, and Catherine Playoust.
\newblock The {M}agma algebra system. {I}. {T}he user language.
\newblock {\em J. Symbolic Comput.}, 24(3-4):235--265, 1997.
\newblock Computational algebra and number theory (London, 1993).

\bibitem{C89}
C.~Cibils.
\newblock Cohomology of incidence algebras and simplicial complexes.
\newblock {\em J. Pure Appl. Algebra}, 56(3):221--232, 1989.

\bibitem{GAP4}
The GAP~Group.
\newblock {\em {GAP -- Groups, Algorithms, and Programming, Version 4.12.2}},
  2022.

\bibitem{GR01}
M.~A. Gatica and M.~J. Redondo.
\newblock Hochschild cohomology and fundamental groups of incidence algebras.
\newblock{ \em Communications in Algebra}, 29:5, 2269--2283, DOI: 10.1081/AGB-100002183, 2001.

\bibitem{GS83}
M.~Gerstenhaber and S.~D. Schack.
\newblock Simplicial cohomology is {H}ochschild cohomology.
\newblock {\em J. Pure Appl. Algebra}, 30(2):143--156, 1983.

\bibitem{H89}
D.~Happel.
\newblock Hochschild cohomology of finite-dimensional algebras.
\newblock volume 1404 of {\em Lecture Notes in Math.}, pages 108--126.
  Springer, Berlin, 1989.

\bibitem{Hoc45}
G.~Hochschild.
\newblock On the cohomology groups of an associative algebra.
\newblock {\em Annals of Mathematics (2)}, 46:58--67, 1945.

\bibitem{IZ90}
K.~Igusa and D.~Zacharia.
\newblock On the cohomology of incidence algebras of partially ordered sets.
\newblock {\em Comm. Algebra}, 18(3):873--887, 1990.


\bibitem{M66}
M.~C. McCord.
\newblock Singular homology groups and homotopy groups of finite topological
  spaces.
\newblock {\em Duke Math. J.}, 33:465--474, 1966.

\bibitem{sagemath}
{W. A. Stein, et al.}
\newblock {\em {S}ageMath, the {S}age {M}athematics {S}oftware {S}ystem,
  ({V}ersion 10.1.beta0), The Sage Development Team}, 2023.
\newblock {\tt https://www.sagemath.org}.

\bibitem{W19}
S.~J. Witherspoon.
\newblock {\em Hochschild cohomology for algebras}, volume 204 of {\em Graduate
  Studies in Mathematics}.
\newblock American Mathematical Society, Providence, RI, [2019] \copyright
  2019.

\end{thebibliography}
\end{document}